\documentclass{amsart}
\usepackage{amsmath}
\usepackage{amssymb}
\usepackage{amsthm}
\usepackage{epsfig}
\usepackage{amsfonts}
\usepackage{amscd}
\usepackage{mathrsfs}
\usepackage{enumerate}
\usepackage{latexsym}
\usepackage{url}

\newtheorem{theorem}{Theorem}

\newtheorem{question}[theorem]{Question}
\newtheorem{conjecture}[theorem]{Conjecture}

\newenvironment{proof of claim}{\noindent\textbf{Proof of the claim.}}{\hfill{$\square$}\newline}

\theoremstyle{definition}

\theoremstyle{remark}

\numberwithin{equation}{section}

\begin{document}


\title[Connectifying a space by adding one point]{Connectifying a space by adding one point}

\author{M.R. Koushesh}

\address{Department of Mathematical Sciences, Isfahan University of Technology, Isfahan 84156--83111, Iran and School of Mathematics, Institute for Research in Fundamental Sciences (IPM), P.O. Box: 19395--5746, Tehran, Iran.}

\email{koushesh@cc.iut.ac.ir}





\begin{abstract}
It is a classical theorem of Alexandroff that a locally compact Hausdorff space has a one-point Hausdorff compactification if and only if it is non-compact. The one-point Hausdorff compactification is indeed obtained by adding the so called ``point at infinity.'' Here we consider the analogous problem of existence of a one-point connectification, and keeping analogy, we prove that a locally connected normal space has a one-point normal connectification if and only if it has no compact component.
\end{abstract}

\maketitle


\section{Introduction}

In 1924, P. Alexandroff proved in \cite{A} that a locally compact non-compact Hausdorff space $X$ can be compactified by adding a ``point at infinity,'' that is, there is a compact Hausdorff space $\alpha X=X\cup\{\infty\}$, where $\infty$ is not in $X$, which contains $X$ as a dense subspace. The analogous question (also attributed to P. Alexandroff) of determining Hausdorff spaces which can be connectified by adding one point remained however unsettled. The question indeed has motivated a significant amount of research so far, with perhaps the earliest serious one dating as far back as 1945. In 1945, B. Knaster proved in \cite{Kn} that a separable metrizable space has a separable metrizable one-point connectification if and only if it can be embedded in a separable metrizable connected space as a proper open subspace. Among recent results we mention that of M. Abry, J.J. Dijkstra, and J. van Mill in \cite{ADvM} which states that a separable metrizable space in which every component is open has a separable metrizable one-point connectification if and only if it has no compact component.

Here in this note we prove that a locally connected normal space has a one-point normal connectification if and only if it has no compact component. This may be considered as a dual to the theorem of Alexandroff that a locally compact Hausdorff space has a one-point Hausdorff compactification if and only if it is non-compact.

Recall that \textit{components} (also called \textit{connected components}) of a space $X$ are the maximal connected subsets of $X$ which partition $X$ into pairwise disjoint non-empty closed subsets. Components of locally connected spaces are open. A space $Y$ is called a \textit{one-point connectification} of a space $X$ if $Y$ is connected, contains $X$ as a dense subspace, and $Y\setminus X$ is a singleton. One-point compactifications are defined analogously. (See the standard text \cite{E} for further information.)

\section{The connectification theorem}

Let us first review the statement of the well known theorem of Alexandroff. This theorem (and its other equivalent forms) indeed motivates the statement of our theorem.

\begin{theorem}[Alexandroff]\label{KLFA}
A locally compact Hausdorff space has a one-point Hausdorff compactification if and only if it is non-compact.
\end{theorem}

Observe that locally compact Hausdorff spaces are $T_{3\frac{1}{2}}$ (completely regular) and compact Hausdorff spaces are $T_4$ (normal). Therefore, the Alexandroff theorem can be reformulated in the following form.

\begin{theorem}[Alexandroff]\label{HFD}
Let $i=2,3,3\frac{1}{2},4$. A locally compact $T_i$-space has a one-point $T_i$-compactification if and only if it is non-compact.
\end{theorem}

Keeping analogy, we expect the above theorem to hold true if we replace ``compactness'' by ``connectedness.'' We state this formally as a conjecture and prove it in the case when $i=4$. The case when $i=3\frac{1}{2}$ also holds and is proved in \cite{Ko}.

\begin{conjecture}\label{JHGH}
Let $i=2,3,3\frac{1}{2},4$. A locally connected $T_i$-space has a one-point $T_i$-connectification if and only if it has no compact component.
\end{conjecture}

\begin{theorem}\label{JJH}
A locally connected normal space has a one-point normal connectification if and only if it has no compact component.
\end{theorem}

\begin{proof}
Let $X$ be a locally connected normal space.

Suppose that $X$ has a one-point normal connectification $Y$. Suppose to the contrary that $X$ has a compact component $C$. Then $C$ is closed in $Y$, as $C$ is compact and $Y$ is Hausdorff. On the other hand, $C$ is also open in $Y$, as $C$ is open in $X$, since $X$ is locally connected, and $X$ is open in $Y$, since $Y$ is a $T_1$-space. Therefore $C$ is a non-empty proper simultaneously closed and open subset of $Y$, which contradicts connectedness of $Y$.

Now, suppose that $X$ has no compact component. First, we show that for every component $C$ of $X$ there is a collection $\mathscr{D}(C)$ of closed sets in $C$ (and therefore in $X$, as components of a space are always closed) with the finite intersection property such that the intersection of the elements of $\mathscr{D}(C)$ is empty, and $\mathscr{D}(C)$ is maximal with respect to this property.

Let $C$ be a component of $X$. Let $\Delta$ be the family consisting of all collections $\mathscr{D}$ of closed sets in $C$ with the finite intersection property and empty intersection. We partially order $\Delta$ with set-theoretic inclusion $\subseteq$. Note that $\Delta$ is non-empty, as by our assumption $C$ is non-compact. To show that $\Delta$ has a maximal element, by Zorn's lemma, it suffices to check that every linearly ordered subfamily of $\Delta$ has an upper bound in $\Delta$. But it is easy to check that for any linearly ordered subfamily $\Lambda$ of $\Delta$ the union $\bigcup\Lambda$ is in $\Delta$. This shows the existence of $\mathscr{D}(C)$.

It follows from the maximality of $\mathscr{D}(C)$ that for any finite subcollection $\mathscr{F}$ of $\mathscr{D}(C)$ the intersection $\bigcap\mathscr{F}$ is also in $\mathscr{D}(C)$, as $\mathscr{D}(C)\cup\{\bigcap\mathscr{F}\}$ is a collection of closed sets in $C$ with the finite intersection property and empty intersection which contains $\mathscr{D}(C)$.

We now proceed with the construction of a one-point connectification for $X$. Take some $p$ outside $X$ and let $Y=X\cup\{p\}$. Let the topology $\mathscr{T}$ of $Y$ consists of (I) all open sets of $X$, together with (II) all subsets $U$ of $Y$ such that
\begin{enumerate}
  \item $p$ is in $U$;
  \item $U\cap X$ is open in $X$;
  \item $U$ contains an element from $\mathscr{D}(C)$ for each component $C$ of $X$.
\end{enumerate}
We verify that this indeed defines a topology on $Y$.

The empty set is of type (I) and the whole set $Y$ is of type (II). We check that the intersection of any two elements of $\mathscr{T}$ is in $\mathscr{T}$. It is clear that the intersection of any two elements of type (I) is of type (I), and the intersection of an element of type (I) with an element of type (II) is of type (I). Let $U$ and $V$ be elements of type (II). It is clear that $p$ is in $U\cap V$ and $(U\cap V)\cap X$ is open in $X$. For a component $C$ of $X$ there are elements $A$ and $B$ of $\mathscr{D}(C)$ contained in $U$ and $V$, respectively. Now, $A\cap B$ is an element of $\mathscr{D}(C)$, as $\mathscr{D}(C)$ has any finite intersection of its elements, and $U\cap V$ contains $A\cap B$. Thus $U\cap V$ is of type (II). The fact that the union of any number of elements of $\mathscr{T}$ is in $\mathscr{T}$ may be checked analogously. Therefore $\mathscr{T}$ is a topology on $Y$.

Note that $X$ is a subspace of $Y$, as any open set in $X$ is of type (I) and any open set in $Y$ intersected with $X$ is open in $X$. Also, $X$ is dense in $Y$, as any neighborhood of $p$ in $Y$ contains an element from $\mathscr{D}(C)$ for a component $C$ of $X$, and elements of $\mathscr{D}(C)$ are necessarily non-empty.

We check that $Y$ is connected. Suppose otherwise that there is a separation $U$ and $V$ for $Y$. Without any loss of generality we may assume that $p$ is in $U$. But then $U$ intersects any component $C$ of $X$, as it contains an element from $\mathscr{D}(C)$, and thus $U$ contains the whole $C$, as $C$ is connected. This implies that $U=Y$, which is a contradiction.

To conclude the proof we check that $Y$ is a normal space. First, we check that $Y$ is a Hausdorff space. Let $y$ and $z$ be distinct elements in $Y$. First, suppose that $y$ and $z$ are both in $X$. Then $y$ and $z$ are separated by disjoint sets which are open in $X$, and therefore open in $Y$. Next, suppose that either $y$ or $z$, say $y$, is $p$. Then $z$ is in $X$. Let $C$ be the component of $X$ which contains $z$. Note that components of $X$ are open in $X$, as $X$ is locally connected. There is an element $A$ of $\mathscr{D}(C)$ which does not contain $z$. Since $X$ is regular, there are disjoint open sets $U$ and $V$ in $X$ containing $z$ and $A$, respectively. The sets $(Y\setminus C)\cup V$ and $U\cap C$ are open in $Y$, are disjoint, and contain $y$ and $z$, respectively.

Next, we check that disjoint closed sets in $Y$ are separated by disjoint open sets in $Y$. Let $F$ and $G$ be disjoint closed sets in $Y$. First, suppose that neither $F$ nor $G$ contains $p$. Then $F$ and $G$ are disjoint closed sets in $X$, and therefore, since $X$ is normal, are separated by disjoint sets which are open in $X$, and thus open in $Y$. Next, suppose that either $F$ or $G$, say $F$, contains $p$. Then $X\setminus G$ is a neighborhood of $p$ in $Y$. For a component $C$ of $X$ let $A_C$ be an element of $\mathscr{D}(C)$ contained in $X\setminus G$. Then $(F\cap X)\cup A_C$ and $G$ are disjoint closed sets in $X$, and therefore, since $X$ is normal, are separated by disjoint open sets in $X$. Let $U_C$ and $V_C$ be disjoint open sets in $X$ containing $(F\cap X)\cup A_C$ and $G$, respectively. The sets
\[\{p\}\cup\bigcup\{C\cap U_C:\mbox{$C$ is a component of $X$}\}\]
and
\[\bigcup\{C\cap V_C:\mbox{$C$ is a component of $X$}\}\]
are disjoint open sets in $Y$ containing $F$ and $G$, respectively.
\end{proof}

We conclude with a question whose statement is indeed motivated by Conjecture \ref{JHGH}.

\begin{question}\label{HGF}
Determine those topological properties $\mathscr{P}$ for which the following holds true: A locally connected space with $\mathscr{P}$ has a one-point connectification with $\mathscr{P}$ if and only if it has no compact component.
\end{question}

Observe that by the theorem of M. Abry, J.J. Dijkstra, and J. van Mill in \cite{ADvM} the statement in the above question holds true when $\mathscr{P}$ is separable metrizability.

\section*{Acknowledgements}

This research was in part supported by a grant from IPM.


\end{document}